\pdfoutput=1
   \documentclass{amsart}
\usepackage {amsmath, amscd}
\usepackage {amssymb}
\usepackage {epsfig}
\usepackage {bm}
\usepackage {indentfirst} 

\usepackage{yhmath}
\usepackage{graphicx}
\usepackage{color}

\usepackage[active]{srcltx}

\numberwithin{equation}{section}

\def\<{\langle}
\def\>{\rangle}

\def\DD{{\mathcal D}}

\def\HH{{\mathcal H}}

\def\KK{{\mathcal K}}
\def\LL{{\mathcal L}}

\def\XX{{\mathcal X}}

\def\bbR{\mathbb{R}}
\def\bbC{\mathbb{C}}

\def\bbD{\mathbb{D}}
\def\bbT{\mathbb{T}}

\newcommand{\Arg}{\mathop{\rm Arg}}

\newcommand{\Wth}{\widehat\theta}

\newcommand{\Hth}{\bm{H}_\theta}

\newtheorem{lemma}{Lemma}[section]
\newtheorem{proposition}[lemma]{Proposition}
\newtheorem{theorem}[lemma]{Theorem}

\theoremstyle{definition}

\newtheorem{example}[lemma]{Example}

\title{The numerical range of a contraction with finite defect numbers}
\author{Hari Bercovici}

\address{Department of Mathematics, Indiana University, Bloomington, IN 47405, USA}
\email{bercovic@indiana.edu}
\author{Dan Timotin}
\address{Simion Stoilow Institute of Mathematics of the Romanian Academy, PO Box 1-764, Bucharest 014700, Romania}
\email{Dan.Timotin@imar.ro}
\date{}


\thanks{HB was supported in part by grants from the National Science Foundation. DT was supported in part by a grant of the Romanian National Authority for Scientific Research, CNCS-UEFISCDI, project number PN-II-ID-PCE-2011-3-0119}
\begin{document}

\begin{abstract}
An $n$-dilation of a contraction $T$ acting on a Hilbert space $\HH$ is a unitary dilation  acting on $\HH\oplus \bbC^n$. We show that if both defect numbers of $T$ are equal to~$n$, then  the closure of the numerical range of $T$ is the intersection of the closures of the numerical ranges of its $n$-dilations. We also obtain detailed information about the geometrical properties of the numerical range of $T$ in case $n=1$.
\end{abstract}

\maketitle

\section{Numerical range and dilations}

Assume that $\HH$ is a complex separable Hilbert space and denote by $\LL(\HH)$ the algebra of all bounded linear operators on~$\HH$. We  also use the notations $\bbD= \{z\in\bbC:|z|<1\}$ and $\bbT= \{z\in\bbC:|z|=1\}$. The spectrum of $T\in\LL(\HH)$ is denoted by $\sigma(T)$, while   the \emph{numerical range} of  $T$ is defined by
\[
W(T):= \{\langle Tx, x\rangle: x\in\HH,\,  \|x\| = 1\}.
\]

In this paper we are only concerned with contractions, that is, operators of norm at most~1. An arbitrary contraction can be decomposed as a direct sum of a unitary operator and a completely nonunitary contraction.
The following basic properties of the numerical range of a contraction can be found, for instance, in~\cite[Ch. 1]{GR}.
\begin{proposition}\label{pr:basic} Let $T\in\LL(\HH)$, $\|T\|\le 1$. Then
$W(T)$ is a  convex subset of $\bbC$ which satisfies:
\begin{enumerate}
\item
$W(T)\subset \overline{\bbD}$ and $W(T)\subset \bbD$ if $T$ is completely nonunitary.
\item
$\sigma(T)\subset\overline{W(T)}$;
\item
 $\overline{W(T)}\cap \bbT=\sigma(T)\cap \bbT$.
\end{enumerate}
\end{proposition}

We denote $D_T=(I-T^*T)^{1/2}$ and $\DD_T=\overline{D_T \HH}$; these are called the \emph{defect operator} and the \emph{defect space} of~$T$ respectively. The  dimensions of $\DD_T$ and $\DD_{T^*}$ are called the \emph{defect indices} of~$T$.

It is well-known that $T$ admits unitary dilations; that is, there exist a space $\KK\supset\HH$ and a unitary operator $U\in\LL(\KK)$ such that $T=PU|\HH$, where $P$ denotes the orthogonal projection in $\KK$ onto $\HH$.  One can always take $\KK$ to be $\HH\oplus\HH$; however, this is not  the optimal choice when the defect spaces of $T$ are of equal finite dimension $n$. Indeed, in this case there exist unitary dilations acting on $\KK=\HH\oplus \bbC^n$, and $n$ is the smallest possible value of $\dim (\KK\ominus\HH)$. We call such dilations \emph{unitary $n$-dilations}.

It is obvious that $W(T)\subset W(U)$ for any unitary dilation $U$ of $T$.
Choi and Li~\cite{ChoiLi} showed that, in fact,
\begin{equation}\label{eq:choi-li}
\overline{W(T)} = \bigcap\{\overline{W(U)}: U \in \LL(\HH \oplus \HH)~\mbox{is a unitary dilation of}~ T\},
\end{equation}
thus answering a question raised by Halmos (see, for example, \cite{Halmos}). We note that when $\HH$ is $m$-dimensional, the construction in~\cite{ChoiLi}  produces  dilations which act on a space of dimension $2m$.

Assume now that $\dim\DD_T=\dim \DD_{T^*}=n$. We prove in Section~\ref{se:unitary dilations} the stronger result
\begin{equation}\label{eq:general intersection formula}
\overline{W(T)} = \bigcap\{\overline{W(U)}: U \in \LL(\KK)~\mbox{is a unitary $n$-dilation of}~ T\},
\end{equation}
that is, we use only the most ``economical'' unitary dilations of~$T$.
This relation was proved earlier in special cases, namely for $\dim\HH<\infty$ and $n=1$ in~\cite{GauWu}, for  $\dim\HH<\infty$ and general $n$ in~\cite{GLW}, and for particular cases with $\dim\HH=\infty$ in~\cite{CGP} and~\cite{BGT}.

The remainder of the paper deals with the special case when $\dim\DD_T=\dim \DD_{T^*}=1$. Partly as a consequence of the results in Section~\ref{se:unitary dilations}, we can investigate in detail the geometric properties of $W(T)$. The study in Section~\ref{se:general defect 1} does not use the functional model of a contraction, but this model is necessary for the  results proved in later sections. Some properties are similar to those proved for  finite dimensional contractions, but new phenomena appear for which we  provide some illustrative examples.

Since the numerical range of a direct sum is the convex hull of the union of the numerical ranges of the summands, and the numerical range of a unitary operator is rather well understood,  we always assume in the sequel that~$T$ is completely nonunitary. Also, if $T$ has different defect indices, then  $W(T)=\mathbb{D}$; so we restrict ourselves to studying completely nonunitary contractions with equal defect indices.

\section{Unitary $n$-dilations}\label{se:unitary dilations}

Let $T$ be a completely nonunitary contraction  with $\dim\DD_T=\dim \DD_{T^*}=n<\infty$. The operator
\begin{equation}\label{eq:definition of tilde T}
\widetilde T= \begin{pmatrix}
T & 0\\  D_T& 0
\end{pmatrix}.
\end{equation}
is a partial isometry on $\HH\oplus\DD_T$ and $\sigma(\widetilde{T})=\sigma(T)\cup \{0\}$. Both  $\ker \widetilde T$ and $\ker \widetilde T^*$  have  dimension $n$, and so any unitary operator $\Omega:\ker \widetilde T\to\ker\widetilde T^*$ determines a unitary $n$-dilation $U_\Omega$ of $T$ by the formula
\begin{equation}\label{eq:definition of U_Omega}
U_\Omega(x)= \begin{cases}
\widetilde{T}x& \text{ if }x\in \ker \widetilde T^\perp\\
\Omega x & \text{ if }x\in \ker \widetilde T.
\end{cases}
\end{equation}
Conversely, any unitary $n$-dilation of $T$ is unitarily equivalent to some $U_\Omega$. The set $\sigma(U_\Omega)\setminus\sigma(T)$ consists of isolated Fredholm eigenvalues of $U_\Omega$.

\begin{lemma}\label{le:extending partial isometry}
Assume that $\mathcal{X}$ is a Hilbert space and $A\in\LL(\XX)$ is a partial isometry such that $\dim\ker A=\dim\ker A^*=d<\infty$. If $\lambda\in\bbT$ is such that $A-\lambda$ is Fredholm and $\dim\ker(A-\lambda)<d$, then there exists a partial isometry $A_1$ with $\dim\ker A_1=\dim\ker A_1^*=d-1$, $A_1|\ker A^\perp=A$, such that $\dim\ker(A_1-\lambda)=\dim\ker(A-\lambda)+1$.
\end{lemma}

\begin{proof}
Observe that the Fredholm index of $A-\lambda$ is equal to zero since $\lambda$ is in the closure of $\mathbb{C}\setminus\sigma(A)$ and so codim$(A-\lambda)\XX<d$. Also, $\ker(A-\lambda)\subset \ker A^\perp$.

If $\ker A\cap \ker A^*\not=\{0\}$, we pick $h\in \ker A\cap \ker A^*$ of norm 1, and define the desired $A_1$ by $A_1|\ker A^\perp=A$,  $A_1 h= \lambda h$, and $A_1w=0$ for $w\in\ker A\cap\{h\}^\perp$.

If $\ker A\cap \ker A^*=\{0\}$, then $\dim (\ker A\vee \ker A^*)=2d$, whence
 there exist a nonzero vector $z\in \ker A\vee \ker A^*$ and $x\in \ker A^\perp$ such that $(A-\lambda)x=z$.

Write $z=y+y_*$, with $y\in\ker A$ and $y_*\in \ker A^*$. We have then
$(A-\lambda)x= y+y_*$
or, equivalently,
\begin{equation}\label{eq:condition lambda}
Ax-y_*=\lambda x+ y.
\end{equation}
On each side of this equality the two terms are mutually orthogonal. Equating the norms and noting that $\|Ax\|=\|x\|$, we obtain $\|y\|=\|y_*\|$, The required $A_1$ is then obtained by adding to the condition $A_1|\ker A^\perp=A$ the relation $A_1 y= - \lambda y_*$ and $A_1w=0$ for $w\in\ker A\cap\{y\}^\perp$. Indeed, we have then by~\eqref{eq:condition lambda}
\[
A_1(x+\bar\lambda y)=Ax- y_*=\lambda x+y=
\lambda(x+\bar\lambda y).
\]
Since $y\not=0$, $x+\bar\lambda y\not\in\ker A^\perp\supset\ker(A-\lambda)$,
and thus $\lambda$ is an eigenvalue of $A_1$ with multiplicity $\dim\ker(A-\lambda)+1$.
\end{proof}

\begin{proposition}\label{pr:choose eigenvalues}
Assume that $\lambda_1, \dots,\lambda_{k}$ are distinct points in $\bbT\setminus \sigma(T)$ and $n_1,n_2,\dots,n_k$ are positive integers satisfying $\sum_{j=1}^k n_j=n$. There exists a unitary $n$-dilation $U$ of $T$ such that $\lambda_j$  is an eigenvalue of $U$ with multiplicity $\ge n_j$, $j=1,2,\dots k$.
\end{proposition}

\begin{proof}
Consider the partial isometry $\widetilde T$ defined by~\eqref{eq:definition of tilde T}. We construct inductively partial isometries $W_0,W_1,\ldots, W_n$ such that
\begin{enumerate}
\item $W_0=\widetilde T$;
\item
$W_{m+1}|(\ker W_m)^\perp= W_{m}|(\ker W_m)^\perp$ for $m=0, 1, \ldots, n-1$;
\item
$\dim\ker W_m=n-m$;
\item
$\sum_{j=1}^k \min\{n_j, \dim\ker (W_m-\lambda_j)\}\ge m$.
\end{enumerate}

Assume that $m<n$ and $W_m$ has been constructed. If the inequality in (4) is strict, we choose $W_{m+1}$ to be an arbitrary partial isometry satisfying (2) and (3) with $m+1$ in place of $m$. If (4) is an equality, then there exists $j\in\{1,\ldots,k\}$ such that $\dim\ker (W_m-\lambda_j)<n_j$. Since $W_m-\lambda_j$ is a finite rank perturbation of the invertible operator $\tilde T-\lambda_j$, it is Fredholm. Lemma~\ref{le:extending partial isometry} applied to $A=W_m$ and $\lambda=\lambda_j$, produces a partial isometry $A_1$; we define $W_{m+1}=A_1$. This completes the induction process; the operator $U=W_n$ satisfies the conclusion of the proposition.
\end{proof}

Proposition~\ref{pr:choose eigenvalues} is stated without proof in~\cite{martin}.

\begin{lemma}\label{le:intersection_num_ranges}
Assume that $U$ is a unitary $n$-dilation of $T$,  $\lambda_1, \dots, \lambda_{n+1}$ are  eigenvalues of $U$ {\rm(}not necessarily distinct, but with repeated values being multiple eigenvalues\/{\rm)}. Then the  convex hull of $\lambda_1, \dots, \lambda_{n+1}$ contains a  point in $W(T)$.
\end{lemma}

\begin{proof}
If $\xi_1,\dots, \xi_{n+1}$ are  corresponding eigenvectors of $U$, then their linear span $Y$ has dimension $n+1$, and thus its intersection with $\HH$ contains a vector $\xi$ of norm 1. Then $\<U\xi,\xi\>=\<T\xi,\xi\>$ belongs to the numerical range of the restriction of $U$ to $Y$, which is precisely the  convex hull of $\lambda_1, \dots, \lambda_{n+1}$.
\end{proof}

\begin{theorem}\label{th:intersection result}
Assume that $T$ is a contraction with $\dim\DD_T=\dim \DD_{T^*}=n<\infty$.
\begin{enumerate}
\item
If   $d$ is a support line for $\overline{W(T)}$, then there exists a unitary $n$-dilation $U$ of $T$ such that $d$ is a support line for $\overline{W(U)}$.
\item
$\overline{W(T)} = \bigcap\{\overline{W(U)}: U \in \LL(\KK)~\mbox{is a unitary $n$-dilation of}~ T\}$.
\end{enumerate}
\end{theorem}

\begin{proof}
Obviously (2) follows from (1), so we only have to prove (1). Assertion (1) is immediate if $d$ is tangent to $\bbT$, so we may assume that $d$ intersects $\bbT$ in two points. These two points must belong to the closure of one of the component arcs $\wideparen{z_1,z_2}$ of $\bbT\setminus\sigma(T)=\bbT\setminus \overline{W(T)}$.

Consider then, for $\epsilon>0$, the line $d_\epsilon$ parallel to $d$ at a distance $\epsilon$ and that does not meet $W(T)$. For $\epsilon$ sufficiently small, $d_\epsilon$ intersects $\wideparen{z_1,z_2}$ in two points $\lambda_\epsilon, \lambda_\epsilon'$. By Proposition~\ref{pr:choose eigenvalues} there exists a unitary $n$-dilation $U_\epsilon$ of $T$ (acting on $\HH\oplus\DD_T$) which has $\lambda_\epsilon$ as an eigenvalue of multiplicity $n$. We assert that $d$ is a support line for $W(U_\epsilon)$. Indeed, if $U$ had an eigenvalue $\lambda'_\epsilon$ on $\bbT\cap \wideparen{z_1,z_2}$, we could apply Lemma~\ref{le:intersection_num_ranges} to the set formed by $\lambda$ repeated $n$ times together with $\lambda'$ to obtain an element  in $W(T)$ on the segment $[\lambda_\epsilon, \lambda_\epsilon']$ which is at distance $\epsilon$ from $W(T)$.

Finally, noting that  the set of unitary $n$-dilations of $T$ on $\HH\oplus\DD_T$ is compact,  a limit point of the unitary dilations $U_\epsilon$ will itself have $d$ as a support line.
\end{proof}

Note that based on the  Sz.-Nagy---Foias functional model~\cite{SNF} one can give a function theoretical description of all unitary $n$-dilations of $T$; this description, originating in~\cite{BL}, is actually used in~\cite{BGT, CGP}. As seen above, it appears that direct arguments are simpler and more transparent. The functional model does however offer some advantages which we start to exploit in Section~\ref{se:functional model}.

\section{Defect spaces of dimension~1}\label{se:general defect 1}

The remainder of the paper  deals with the particular case
when $\dim\DD_T=\dim \DD_{T^*}=1$ and  $T$ is completely nonunitary; to avoid trivialities, we will assume in the sequel that $\dim\HH>1$.
The first condition is of course the simplest case of finite and equal defect numbers, and this allows us to obtain several geometrical properties of $W(T)$, extending results obtained for finite dimensional matrices in~\cite{GauWu, GauWu2, GauWu3} and for certain infinite dimensional operators in~\cite{CGP, CGP2}.

Recall that a point of a convex set $C$ in the plane is said to be \emph{exposed} if there is a support line for $C$ that intersects it only in that point; such a line is called an \emph{exposing} line. A \emph{corner} is an exposed point for which the exposing line is not unique.
If $T$ is a completely nonunitary contraction, then all points in $\overline{W(T)}\cap\bbT$ are obviously exposed. We describe next some other properties of the boundary $\partial W(T)$. 

\begin{theorem}\label{th:exposed and corners}
Let $T\in\LL(\HH)$ be a completely nonunitary contraction with $\dim\DD_T=\dim \DD_{T^*}=1<\dim \HH$.
\begin{enumerate}
\item If $w\in\bbD$ is an exposed point of $\overline{W(T)}$, then $w\in W(T)$.
\item Any eigenvalue of $T$ is in the interior of $W(T)$.
\item $\overline{W(T)}$ has no corners in $\bbD$.
\end{enumerate}
\end{theorem}

\begin{proof}
(1) Assume that $d$ is an exposing line for $w$. If $\alpha\in d\cap \bbT$ then $\alpha\not\in\sigma(T)$, since otherwise we would have $[w,\alpha]\subset\overline{W(T)}$.  Thus $d$ intersects $\bbT$ in two points $\alpha, \alpha'$, neither of which  belongs to $\sigma(T)$. Theorem~\ref{th:intersection result} yields a unitary 1-dilation $U$ of $T$ whose numerical range also has $d$ as a support line. Then $\alpha$ and $\alpha'$ are eigenvalues of $U$, and  Lemma~\ref{le:intersection_num_ranges} implies that $w\in W(T)$.

(2) It was shown in~\cite{Sin}  that if an eigenvalue $\alpha$ is in the boundary of the numerical range, then the corresponding eigenspace reduces $T$. Thus $T=\alpha\oplus T'$, and $\dim \DD_{T'}=\dim\DD_T-1=0$. Since, similarly, $\dim\DD_{T'{}^*}=0$, $T'$ is both unitary and completely nonunitary, which is possible only when $T'$ acts on the trivial space $\{0\}$. This contradicts the assumption $\dim\HH>1$.

(3) If $w$ is a corner of $\overline{W(T)}$ and $|w|<1$, then there exist a support line of $W(T)$ that passes through $w$ and intersects the unit circle in two points in $\bbT\setminus \sigma(T)$. As shown above, it follows then that $w$ actually belongs to $W(T)$. By a result in \cite{Don}, $w$ has to be an eigenvalue of $T$, contradicting~(2).
\end{proof}

Assertions (2) and (3) of Theorem~\ref{th:exposed and corners} were proved  in~\cite{CGP2} in  case the characteristic function of $T$ is a Blaschke product.

It was shown in~\cite{GauWu}  that, when $T$ acts on a finite dimensional space and $U$ is a unitary 1-dilation of $T$, the boundary of the numerical range of~$U$ is a polygon whose $n+1$ sides are tangent to $W(T)$. The theorem below is an infinite dimensional version of this result (see also~\cite{CGP} for a more restricted infinite dimensional extension). The following sections show that additional information can be obtained under certain conditions.
\begin{theorem}\label{th:polygon}
Let $T$ be a completely nonunitary contraction with $\dim\DD_T=\dim \DD_{T^*}=1<\dim\HH$, $U$  a unitary $1$-dilation of $T$, and let the arc $I=\wideparen{z_1z_2}$ be  a connected component of $\bbT\setminus \sigma(T)$. Then:
\begin{enumerate}
\item If the segment $[z_1, z_2]$ is a support line for $W(T)$, then $\sigma(U)\cap I$ contains at most one eigenvalue of $U$.

\item If the segment $[z_1, z_2]$ is not a support line for $W(T)$, then
\begin{enumerate}
\item the only possible limit points of $\sigma(U)\cap I$ are $z_1$ and $z_2$;
\item
each segment determined by consecutive values of $\sigma(U)\cap I$ on $\bbT$ is a support line for $W(T)$, and its intersection with $\overline{W(T)}$ is actually in $W(T)$.
\end{enumerate}

\end{enumerate}
\end{theorem}

\begin{proof}
If the segment $[z_1, z_2]$ is a support line for $W(T)$, then it separates $W(T)$ from $I$. If $\sigma(U)\cap I$ contained two eigenvalues, Lemma~\ref{le:intersection_num_ranges} would imply the existence of a point in $W(T)$ on the segment joining them, which is a contradiction.

Assume then that  the segment $[z_1, z_2]$ is not a support line for $W(T)$. It is immediate that the eigenvalues of $U$ cannot have a limit point in the interior of the arc $I$, since that point would also be a limit point of $W(T)$. Since $W(T)\subset W(U)$, there exist points in $W(U)$ contained in the open arc $\wideparen{z_1z_2}$, and they form a set of eigenvalues of $U$ which is at most countable. If, for two such eigenvalues $\lambda, \lambda'$,  $\wideparen{\lambda\lambda'}\cap \sigma(U)=\emptyset$, then the segment $[\lambda, \lambda']$ is a support line for $W(T)$. Indeed, it must have a common point with $W(T)$ by Lemma~\ref{le:intersection_num_ranges}. On the other hand, if there are points in the circular segment determined by $\lambda$ and $\lambda'$ that belong to the boundary of $W(T)$, those points should belong to $W(U)$. Since $W(U)$ is the closed convex hull of $\sigma(U)$, one must have points in $\wideparen{\lambda\lambda'}\cap \sigma(U)$---a contradiction.
\end{proof}

If $\HH$ is finite dimensional, then $\sigma(T)\cap\bbT=\emptyset$, so $\bbT\setminus \sigma(T)=\bbT$. The analogue of Theorem~\ref{th:polygon}(2b) for this case
 is contained in~\cite[Theorem 6.3]{GauWu2}.

We end this section with a preliminary result that will be used in Section~\ref{se:segments}.

\begin{lemma}\label{le:eigenvalues}
Let $T$ be a completely nonunitary contraction with $\dim\DD_T=\dim \DD_{T^*}=1$ and $U$  a unitary $1$-dilation of $T$. Assume that $d$ is a support line for $W(T)$ as well as for $W(U)$, and $d\cap \bbT=\{\alpha_1,\alpha_2\}$, with $\alpha_1\not=\alpha_2$. Then $d\cap W(T)$ is nonempty if and only if  $\alpha_1$ and $\alpha_2$ are eigenvalues of $U$. When this happens, $d\cap W(T)$ contains exactly one point, and the set $\{x:\<T x,x\>=w\|x\|^2\}$ is a subspace of dimension one.
\end{lemma}

\begin{proof}
Suppose $w\in d\cap W(T)$; then $w=\<T\xi, \xi\>=\<U\xi,\xi\>$ for some $\xi$ with $\|\xi\|=1$. Using the spectral theorem for $U$, it is easy to see that $\alpha_1, \alpha_2$ must be eigenvalues of~$U$, and $\xi$ is a linear combination of corresponding eigenvectors $\xi_1$ and $\xi_2$.

We claim that  $\xi_1$ and $ \xi_2$ are uniquely determined (up to a scalar). Indeed, if $U\xi_1=\alpha_1\xi_1$ and $U\xi_1'=\alpha_1\xi_1'$ with $\xi_1\perp \xi_1'$, then Lemma~\ref{le:intersection_num_ranges} implies that $\alpha_1\in W(T)$, contrary to Proposition~\ref{pr:basic}(1).

Assume that there exist two distinct points $w,w'\in\ell\cap W(T)$; then $w=\<T\xi, \xi\>$ and $w'=\<T\xi', \xi'\>$, and $\xi,\xi'$ are two linearly independent vectors in the linear span of $\xi_1$ and $\xi_2$. It would follow that $\xi_1, \xi_2$ are in the linear span of $\xi,\xi'$, and thus in $\HH$. This contradicts again Proposition~\ref{pr:basic}(1). A similar argument shows that $\{x:\<S_\theta x,x\>=w\|x\|^2\}$ cannot contain two linearly independent vectors.
\end{proof}

\section{The functional model}\label{se:functional model}

The theory of the characteristic function of a completely nonunitary contraction can be used in order to obtain more precise results about its numerical range. We recall briefly the main elements of this theory in the scalar case; the general situation can be found in~\cite{SNF}.

We denote by $H^2$ and $H^\infty$ the usual  Hardy spaces in the unit disc, while $L^2=L^2(\bbT)$. Given $\theta\in H^\infty$ with $\|\theta\|_\infty\le 1$,  denote $\Delta(e^{it})=(1-|\theta(e^{it})|^2)^{1/2}$. The \emph{model space} associated to $\theta$ is defined by
\[
\bm{H}_\theta:=(H^2\oplus \overline{\Delta L^2}) \ominus \{\theta f\oplus \Delta f: f\in H^2\},
\]
and the model operator $S_\theta$ is the compression to $\bm{H}_\theta$ of multiplication by $e^{it}$ on $H^2\oplus \overline{\Delta L^2}$. This \emph{functional model} becomes simpler for inner functions $\theta$, in which case $\Delta=0$ and $\Hth=H^2\ominus \theta H^2$.

It is known (see~\cite{Liv, Moel} for  $\theta$ inner and~\cite{SNF} for the general case) that $\sigma(S_\theta)\cap\bbD$ consists of the zeros of $\theta$ and they are all eigenvalues, while
the set $\bbT\setminus \sigma(T)$ is the union of the subarcs $I$ of $\bbT$ accross which $\theta$ can be extended as an analytic function with $|\theta|=1$ on $I$. We   use the notation $\widehat\theta(z)=z\theta(z)$; this is also an analytic function in the neighborhood of each of the connected components of $\bbT\setminus \sigma(S_\theta)$.

The theory of Sz-Nagy---Foias~\cite{SNF}  says that $S_\theta$ is a completely nonunitary contraction, and any completely nonunitary contraction $T$ with $\dim\DD_T=\dim \DD_{T^*}=1$ is unitarily equivalent to some $S_\theta$. In this case $\theta$ is called the \emph{characteristic function} of~$T$; it is inner precisely when $T^n\to0$ strongly. Since the numerical range is a unitary invariant,  in the sequel  we  study the operator $S_\theta$ for a given function~$\theta$.

\begin{theorem}\label{th:facts theta}
All unitary dilations of $S_\theta$ to $\HH\oplus\bbC$  can be indexed by a parameter $\lambda\in \bbT$, in such a way that a point $\alpha\in\bbT$ is an eigenvalue of $U_\lambda$ if and only if $\Wth$ has an angular derivative in the sense of Carath\'eodory at $\alpha$ and $\Wth(\alpha)=\lambda$. In particular, $\sigma(U_\lambda)\setminus\sigma(S_\theta)$ consists precisely of the solutions in $\bbT\setminus \sigma(S_\theta)$ of the equation $\widehat\theta(z)=\lambda$. 

We have then
\begin{equation}\label{eq:facts theta}
\overline{W(S_\theta)}=\bigcap_{\lambda\in\bbT} \overline{W(U_\lambda)}.
\end{equation}
\end{theorem}

\begin{proof}
Most of the assertions of the theorem are proved in~\cite{Cl} for $\theta$ inner and in~\cite{BL} for general $\theta$. The only exception is formula~\eqref{eq:facts theta}, which is obtained by taking $n=1$ in Theorem~\ref{th:intersection result}(2). 
\end{proof}

Note that $\widehat\theta(z)=\lambda$ might have no solutions in some (or any) component of $\bbT\setminus \sigma(S_\theta)$; it follows that $U_\lambda$ has no eigenvalues therein. If there is $\lambda\in\bbT$ such that the equation $\widehat\theta(z)=\lambda$ has no solutions in the whole set $\bbT\setminus \sigma(S_\theta)$, then
$
\overline{W(S_\theta)}
$
is precisely the closed convex hull of $\sigma(S_\theta)\cap\bbT$.

The space $\Hth$ is finite dimensional if and only if $\theta$ is a finite Blaschke product. This is the case studied extensively by Gau and Wu (see~\cite{GauWu, GauWu2}) and we exclude it in the sequel. Then $S_\theta$ acts on an infinite dimensional space, and $ \overline{W(S_\theta)}\cap\bbT=\sigma(S_\theta)\cap \bbT\not=\emptyset$.

In the remainder of this section we discuss the behavior of $\Wth$ at an endpoint of a component of $\bbT\setminus \sigma(S_\theta)$. If $I=\{e^{it}: t_1< t< t_2\}$ is such a connected component, then $\Wth$ can be extended as an analytic function across $I$ to an open set $O_I\supset I$; we  still denote by $\Wth$ the extended function. We can define a continuous function $\psi_I:(t_1, t_2)\to\bbR$ by the formula $\psi_I(t)=\Arg\widehat\theta(e^{it})=-i\log\widehat\theta(e^{it})$, where the choice of the branch of the logarithm is such as to insure the continuity of $\psi_I$. We have also
\begin{equation}\label{eq:derivative of psi_I}
\psi_I'(t)=e^{it} \frac{{\Wth}'(e^{it})}{\Wth(e^{it})}.
\end{equation}

The following two propositions state essentially that  the nontangential behavior of $\widehat\theta$ is closely related to its behavior along the arc~$I$. They are certainly known, but we have not found an appropriate reference, so we include the proofs for completeness. In a different context, a discussion of the behavior of an analytic function in the neighborhood of an isolated singularity can be found in~\cite{CGP3}.

\begin{proposition}\label{pr:limit of the function}
With the above notations, the following are equivalent:
\begin{enumerate}
\item $\widehat\theta$  has a nontangential limit at $e^{it_1}$.

\item  $\lim_{t\searrow t_1}\psi_I(t)>-\infty$.

\item There exists $\delta>0$ such that $\psi_I$ is one-to-one on $(t_1, t_1+\delta)$, and $\psi_I$  can be continuously extended  to $t_1$.
\end{enumerate}

When these conditions are not satisfied,  the equation $\widehat\theta(z)=\lambda$ has, for each $\lambda\in\bbT$, an infinite sequence of roots in $I$ converging to $e^{it_1}$.
\end{proposition}
\begin{proof}
It is immediate that (2) is equivalent to (3). On the other hand, (2) states the existence of the limit of $\Wth$ at $e^{it_1}$ along the arc $I$. The equivalence between (2) and (1) is then a consequence of Lindel\"of's theorem~\cite{Lin} (see \cite[Theorem 2.3.1]{CoLo} for a more recent presentation) applied to $\Wth$ defined on $O_I$.

Finally, since $\psi_I$ is increasing,  if (2) is not satisfied, then $\lim_{t\searrow t_1}\psi_I(t)=-\infty$. The last assertion is then immediate.
\end{proof}

\begin{proposition}\label{pr:limit of the derivative}
Assume \emph{(1)--(3)} in Proposition~\emph{\ref{pr:limit of the function}} are true. Then the following are equivalent:

\begin{enumerate}
\item
$\widehat\theta$ has an angular derivative in the sense of Carath\'eodory at $e^{it_1}$.
\item
$\psi_I$ is right differentiable at $t_1$.
\end{enumerate}

When these conditions are not satisfied,  $\psi_I$ has an infinite right derivative at~$t_1$.
\end{proposition}

\begin{proof}
Write
\begin{equation}\label{eq:decomposition of function}
\Wth(z)=B(z)\exp\left(-\int_{[0,2\pi)} \frac{e^{it}+z}{e^{it}-z}\,d\nu(t)\right)
\end{equation}
where $B$ is a Blaschke product with zeros $(a_n)$ and $\nu$ is a positive measure. According to a result of Ahern and Clark~\cite{AC1, AC2},  $\Wth$ has an angular derivative at $\zeta\in\bbT$ if and only if
\begin{equation}\label{eq:angular derivative}
\sum_n\frac{1-|a_n|^2}{|\zeta-a_n|^2} + \int_{[0,2\pi)} \frac{d\nu(s)}{|e^{is}-\zeta|^2} <\infty.
\end{equation}

Both conditions (1) and (2) are true for a product if and only if they are true for each of the factors, so we  discuss separately the two factors in~\eqref{eq:decomposition of function}. The same argument implies that we may assume that $a_n=r_n e^{is_n}$, with $r_n\to 1$, $s_n\to t_1$, and $|s_n-t_1|<\pi/4$.

For the Blaschke product $B$,~\eqref{eq:derivative of psi_I} yields
\begin{equation}\label{eq:derivative of psi for blaschke}
\psi_I'(t)=\sum_n \frac{1-|a_n|^2}{|e^{it}-a_n|^2}.
\end{equation}
In particular, if $t_2>t_1$,
\[
\psi_I(t_2)-\psi_I(t_2)=\sum_n \int_{t_1}^{t_2}\frac{1-|a_n|^2}{|e^{it}-a_n|^2}\,dt,
\]
and the value of each integral is the angle at $a_n$ subtended by the arc $\wideparen{e^{it_1}e^{it_2}}$~\cite[p. 41]{Gar}. This angle is bounded below if $s_n>t_1$; as $\psi_I$ is  bounded on $I$, we may assume (discarding some of the zeros, if necessary)
 that   $s_n \le t_1$ for all~$n$. Then, for $t-t_1>0$ and sufficiently small,
\[
\frac{1-|a_n|^2}{|e^{it}-a_n|^2}\le \frac{1-|a_n|^2}{|e^{it_1}-a_n|^2}.
\]
Comparing then~\eqref{eq:angular derivative} (for $\zeta=e^{it_1}$) and~\eqref{eq:derivative of psi for blaschke}, the equivalence of (1) and (2) follows in this case; moreover
\begin{equation}\label{eq:value of psi'(t_1) blaschke}
(\psi_I')_+(t_1)=\lim_{t\searrow t_1} \psi_I'(t)= \sum_n\frac{1-|a_n|^2}{|e^{it_1}-a_n|^2}.
\end{equation}

If $\Wth$ is the singular factor in~\eqref{eq:decomposition of function}, then~\eqref{eq:derivative of psi_I} yields
\begin{equation}\label{eq:derivative of psi for singular}
\psi_I'(t)= 2 \int_{[0,2\pi)} \frac{d\nu(s)}{|e^{is}-e^{it}|^2}.
\end{equation}
For $t-t_1>0$  sufficiently small,
\[
\frac{d\nu(s)}{|e^{is}-e^{it}|^2} \ge \frac{d\nu(s)}{|e^{is}-e^{it_1}|^2},
\]
and again, comparing ~\eqref{eq:angular derivative} (for $\zeta=e^{it_1}$) and~\eqref{eq:derivative of psi for blaschke}, the equivalence of (1) and (2) follows. Moreover,
\begin{equation}\label{eq:value of psi'(t_1) singular}
(\psi_I')_+(t_1)=\lim_{t\searrow t_1} \psi_I'(t)=  2 \int_{[0,2\pi)} \frac{d\nu(s)}{|e^{is}-e^{it_1}|^2}.
 \end{equation}
The last assertion of the proposition follows from~\eqref{eq:value of psi'(t_1) blaschke} and ~\eqref{eq:derivative of psi for singular}.
\end{proof}

We have stated Propositions~\ref{pr:limit of the function} and~\ref{pr:limit of the derivative} for the left endpoint of $(t_1,t_2)$; similar statements are valid for~$t_2$. Note that if $e^{it_1}$ is an isolated point of $\bbT\cap\sigma(T)$ and is thus the endpoint of two components of $\bbT\setminus\sigma(T)$, then the equivalent assertions of Proposition~\ref{pr:limit of the function} cannot be true on both sides of $e^{it_2}$.

Finally, let $I=\{e^{it}: t_1< t< t_2\}$ be, as above, a connected component of $\bbT\setminus\sigma(S_\theta)$. The set
\[
I'=\{t\in(t_1, t_2): \text{ there exists }s\in(t,t_2), \text{ such that }\widehat\theta(e^{is})= \widehat\theta(e^{it})\}
\]
is either empty or equal to an open interval $(t_1, t_1')$.
We define the function $\tau_I:(t_1, t_1')\to (t_1, t_2)$ by
\begin{equation}\label{eq:definition tau}
\tau_I(t)=\min \{s>t:\widehat\theta(e^{is})= \widehat\theta(e^{it})\}.
\end{equation}
Then $\tau_I$ is strictly increasing and differentiable on $(t_1, t_1')$. It follows from Theorem~\ref{th:polygon} that for any $t\in (t_1, t_1')$ the line determined by $e^{it}$ and $e^{i\tau_I(t)}$ is a support line for $W(S_\theta)$, and that any support line for $W(S_\theta)$ that intersects $\bbT$ in two points in $\bbT\setminus \sigma(S_\theta)$ is of this form for some $I$ and $t$.

\section{Segments in the boundary}\label{se:segments}

In this section we  fix a function $\theta\in H^\infty$ with $\|\theta\|_\infty\le1$.
In order to investigate the boundary of $W(T)$, we  discuss separately its intersection with each of the circular segments that correspond to connected components of $\bbT\setminus \sigma(S_\theta)$. Theorem~\ref{th:facts theta}  shows that this study depends on the behavior of $\widehat\theta$ on such an interval.

We already know by Theorem~\ref{th:exposed and corners} that the boundary of $W(S_\theta)$ has no corners in the open disk $\bbD$ and contains no eigenvalues of $S_\theta$. It was proved in~\cite{Ber}, for $\theta$ a finite Blaschke product, that the boundary does not contain straight line segments. As we show in this section, this is not true for all inner functions.

For $w_1\ne w_2$, we  call $[w_1,w_2]$ a \emph{maximal} segment of $\partial W(S_\theta)$ if it is not contained strictly in a larger segment included in $\partial W(S_\theta)$.
\begin{theorem}\label{th:segments in D}
There are no maximal segments of $\partial W(S_\theta)$ contained in $\bbD$.
\end{theorem}

\begin{proof}
Assume $[w_1, w_2]$ is a maximal segment of $\partial W(S_\theta)$ contained in $\bbD$. Denote by $e^{is_1}$ and $e^{is_2}$, $s_1<s_2$  the intersections of the support line $d$ of this segment with $\bbT$ (with the points $e^{is_1}, w_1, w_2, e^{is_2}$ in this order). Then
$e^{is_1}$ and $e^{is_2}$  belong to the same connected component $I$ of $\bbT\setminus\sigma(T)$; the function $\tau_I$ is defined in $s_1$ and $\tau_I(s_1)=s_2$. 

\begin{figure}[h]
\begin{center}
\def\svgwidth{3in}
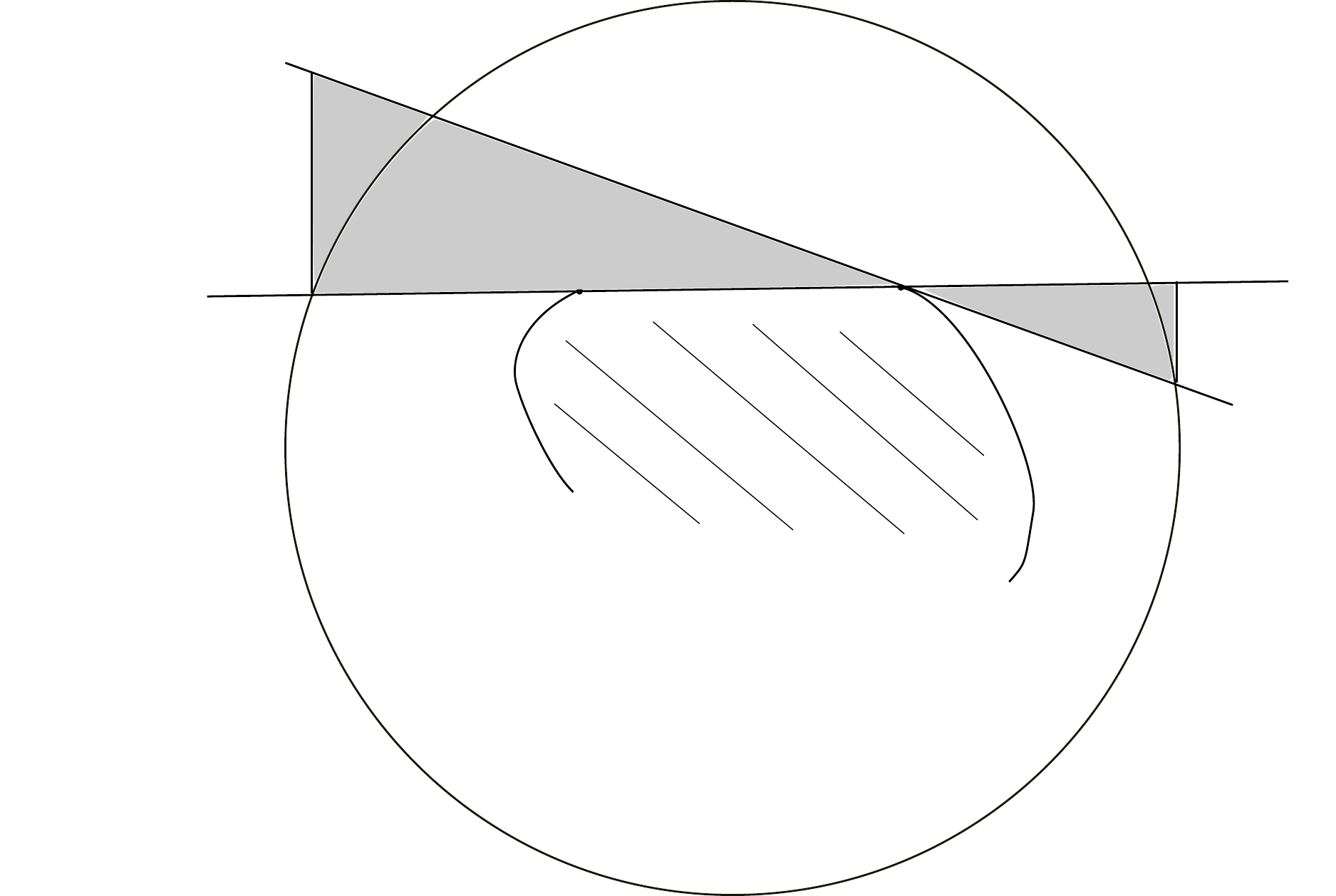
\caption{}
\end{center}
\end{figure}

A geometrical argument shows that the left hand derivative of $\tau_I$ in $s_1$ may be estimated by using that, for values of $s$ slightly smaller than $s_1$, the support line of $W(S_\theta)$ which passes through $e^{is}$ is close to the line $w_1e^{is}$. A computation involving similar triangles (see Figure~1) yields then
\[
\tau'_I{}_-(s_1)=\frac{|e^{is_2}-w_1| }{|e^{is_1}-w_1|}.
\]
A similar computation, using the fact that the right hand derivative of $\tau_{I} $ in $s_1$ can be estimated by using, for $s$ slightly larger than $s_1$, the line $w_2e^{is}$, yields
\[
\tau_{I}{}'_+(s_1)=\frac{|e^{is_2}-w_2| }{ |e^{is_1}-w_2|}.
\]
Since $\tau_{I} $ is differentiable at $s_1$, we  obtain a contradiction if $w_1\not= w_2$.
\end{proof}

Next we discuss the possible segments contained in $\overline{W(T)}$ which  have an endpoint on $\bbT$. By Proposition~\ref{pr:basic}(3) this endpoint, which we  denote by $\alpha_1$, must belong to $\bbT\cap \sigma(S_\theta)$; moreover, it has actually to be in the boundary (relative to $\bbT$) of this last set.

We consider then an arc $I=\wideparen{\alpha_1\alpha_2}$ that is a connected component of $\bbT\setminus \sigma (S_\theta)$. One case can be settled easily.

\begin{theorem}\label{th:one-to-one on the arc}
 $[\alpha_1,\alpha_2]\subset\partial W(S_\theta)$ if and only if $\widehat\theta$ is one-to-one on  $\wideparen{\alpha_1\alpha_2}$. When these conditions are satisfied, the intersection $[\alpha_1,\alpha_2]\cap W(S_\theta)$ is either empty or contains  one point. The second case happens when $\Wth$ has an angular derivative in the sense of Carath\'eodory at both $\alpha_1$ and $\alpha_2$.
\end{theorem}

\begin{proof}
If $\widehat\theta(\beta_1)=\widehat\theta(\beta_2)=\lambda$ for some  $\beta_1\not=\beta_2\in \wideparen{\alpha_1\alpha_2}$, then $\beta_1$ and $\beta_2$ are eigenvalues of $U_\lambda$, and therefore there exists an element of $W(S_\theta)$ on $(\beta_1\beta_2)$ by Lemma~\ref{le:intersection_num_ranges}. Therefore $[\alpha_1,\alpha_2]$ canot be in the boundary of $W(S_\theta)$.

Conversely, suppose $\widehat\theta$ is one-to-one on  $\wideparen{\alpha_1\alpha_2}$.
If $\alpha\in\wideparen{\alpha_1\alpha_2}$, then $\alpha$ is the only solution on that arc of the equation $
\widehat\theta(z)=\widehat\theta(\alpha)$, and thus it is the only eigenvalue therein of the unitary dilation $U_{\widehat\theta(\alpha)}$. The intersection of $\overline{W(U_{\widehat\theta(\alpha)})}$ with the circular segment determined by $\wideparen{\alpha_1\alpha_2}$ is therefore the triangle $\alpha\alpha_1\alpha_2$.  The intersection of all these triangles for $\alpha\in\wideparen{\alpha_1\alpha_2}$ does not contain any point from the (open) circular segment determined by $\alpha_1$ and $\alpha_2$, and therefore Theorem~\ref{th:facts theta} implies that $[\alpha_1,\alpha_2]\subset\partial W(S_\theta)$.

By Theorem~\ref{th:intersection result}(1), there exists a unitary dilation $U_\lambda$ of $S_\theta$ that has $[\alpha_1,\alpha_2]$ as a support line, and Lemma~\ref{le:eigenvalues} implies  that $[\alpha_1,\alpha_2]\cap W(S_\theta)\not=\emptyset$ precisely when $\alpha_1$ and $\alpha_2$ are both eigenvalues of $U_\lambda$. By Theorem~\ref{th:facts theta}, this happens when $\Wth$ has an angular derivative in the sense of Carath\'eodory at $\alpha_1$ and $\alpha_2$.
\end{proof}

We are then left with the case when precisely one of the endpoints of the segment belongs to $\bbT$. A preliminary result discusses corners of $\overline{W(S_\theta)}$ on $\bbT$.

\begin{theorem}\label{th:angle at end point}
If $I=\wideparen{\alpha_1\alpha_2}$ is a connected component of $\bbT\setminus \sigma (S_\theta)$, then the following are equivalent:
\begin{enumerate}
\item There exists a point $\beta\in\wideparen{\alpha_1\alpha_2}$ such that $\widehat\theta$ is one-to-one on  $\wideparen{\alpha_1\beta}$.
\item
There exists a line segment $[\alpha_1,w]$, with $w\in \bbD$, which is a support line for $\partial W(S_\theta)$.
\end{enumerate}
\end{theorem}

\begin{proof}
Assume that (1) is true. It is enough to consider the case when $\widehat\theta$ is not one-to-one on the whole arc $\wideparen{\alpha_1\beta}$, since otherwise we can apply Theorem~\ref{th:one-to-one on the arc}. The assumptions of Proposition~\ref{pr:limit of the function} are then satisfied. The function $\widehat\theta(z)$ has a limit $\eta\in \bbT$ when $z\to \alpha_1$ on the arc $\wideparen{\alpha_1\beta}$, and there exists another point $\beta_1\in \wideparen{\alpha_1\beta}$ such that $\widehat\theta(\beta_1)=\eta$. Moreover, if $\tau_I$ is defined as in~\eqref{eq:definition tau}, then  each of the segments $[e^{it},e^{i\tau_{I} (t)}]$ is a support line for $W(S_\theta)$, and thus the same is true for their limit, which is $[\alpha_1,\beta_1]$. Thus (2) is true.

On the other hand, if (1) is not true, it follows that for any $\lambda\in\bbT$ the equation $\widehat\theta(z)=\lambda$ has an infinite number of solutions in any interval $\wideparen{\alpha_1\beta}$. Then Theorem~\ref{th:facts theta} implies that any interval $\wideparen{\alpha_1\beta}$ contains segments which are support lines for $W(S_\theta)$, so~(2) cannot be true.
\end{proof}

As noted above, if $\alpha_1\in\sigma(S_\theta)\cap \bbT$ and is an isolated singularity, then it is not possible for $\widehat\theta$ to be one to one on a left as well as on a right neighborhood of $\alpha$ (on $\bbT$). Thus, if a point of $\bbT$ is a corner of $\overline{W(T)}$, then one of the support lines is tangent to~$\bbT$.

We can say more about the situation in Theorem~\ref{th:angle at end point}.

\begin{theorem}\label{th:end points}
Suppose  there exists a point $\beta\in\wideparen{\alpha_1\alpha_2}$ such that $\Wth$ is one-to-one on  $\wideparen{\alpha_1\beta}$. The following are equivalent:
\begin{enumerate}
\item
There exists a segment $[\alpha_1,w]\subset \partial W(S_\theta)$ ($w\in \bbD$).
\item
$\widehat\theta$ has an angular derivative in the sense of Carath\'eodory at~$\alpha_1$.
\end{enumerate}
When these equivalent conditions are satisfied, the intersection $[\alpha_1,w]\cap W(S_\theta)$ contains exactly one point.
\end{theorem}

\begin{proof}
Condition (2) above means that $\Wth$ satisfies the assumptions of Proposition~\ref{pr:limit of the derivative}. Using the notations therein (in particular, $\alpha_1=e^{it_1}$), it is easy to see that  the required differentiability of $\psi_I$ at $t_1$ is equivalent to the differentiability of $\tau_I$, or to the boundedness of the derivative of $\tau_I$ on a right neighborhood of~$t_1$.

To show that $\partial W(S_\theta)$ contains a segment with an end at $\alpha$, one has, by Theorem~\ref{th:facts theta}, to prove that  the intersection of all half-planes determined by $[\zeta,\tau_{I} (\zeta)]$, with $\zeta\in  \wideparen{\alpha_1\beta}$, and containing $W(S_\theta)$ contains such a segment; a moment's reflection shows that it is actually enough to consider $\zeta$ in a neighborhood of $\alpha_1$.
\begin{figure}[h]
\begin{center}
\def\svgwidth{3in}
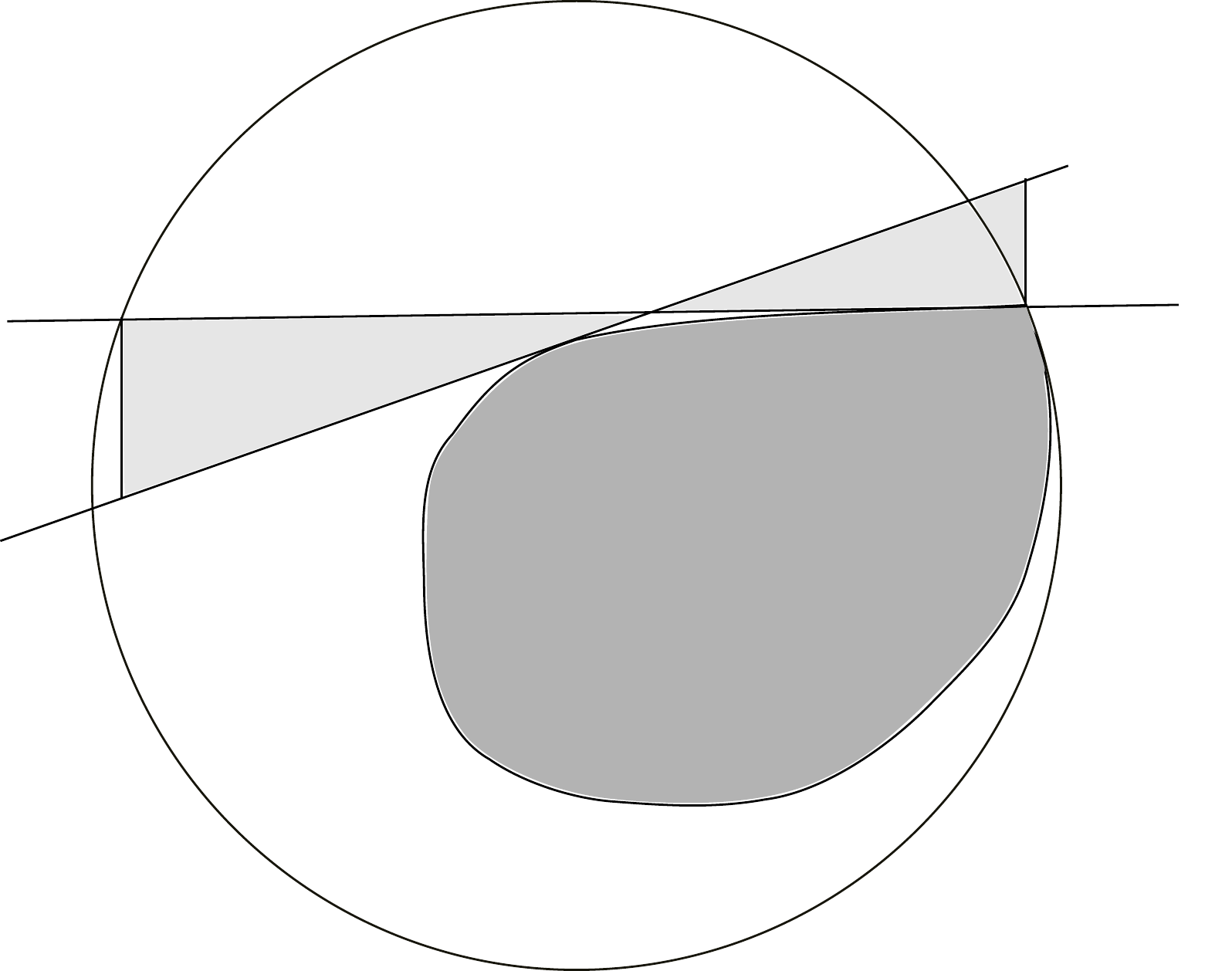
\caption{}
\end{center}
\end{figure}

The argument now is close to the one in Theorem~\ref{th:segments in D}. Remembering that $\beta_1\in\wideparen{\alpha_1\beta}$ satisfies $\Wth(\beta_1)=\Wth(\alpha_1)$,   take $\zeta=e^{it}\in  \wideparen{\alpha_1\beta}$, and denote by $w_t$ the intersection between the segments $[\alpha_1,\beta_1]$ and $[e^{it},e^{i\tau_{I} (t)}]$. Again by considering similar triangles (see Figure~2) we obtain that, when $\zeta$ is close to $\alpha_1$, we have
\[
\frac{|\beta_1- w_t|}{|\alpha_1- w_t|} \approx
\frac{|\beta_1-e^{i\tau_{I} (t)}|}{|\alpha_1-e^{it}|}.
\]
Condition (1) is equivalent to the fact that $|\alpha_1 w_t|$ is bounded below in a neighborhood of $t_1$. The above estimate shows that this happens if and only if the quotient
 $|\beta_1e^{i\tau_{I} (t)}|/|\alpha_1e^{it}|$ is bounded. The last assertion is obviously equivalent to the boundedness of the derivative  of $\tau_{I} $ or, as noted above, to condition (2).
 
Assume now that the equivalent assertions of the theorem are satisfied.  If $\eta=\Wth(\alpha_1)$, then $\Wth(\beta_1)=\eta$, and it follows from Theorem~\ref{th:facts theta} that $\alpha_1$ and $\beta_1$ are eigenvalues of $U_\eta$. Applying then Lemma~\ref{le:eigenvalues} shows that $[\alpha_1,\beta_1]\cap W(S_\theta)$ contains exactly one point.
\end{proof}

With some supplementary effort one can show that, when it exists, the point in $[\alpha_1,w]\cap W(S_\theta)$ is precisely~$w$, the endpoint of the maximal segment contained in~$\partial W(T)$.

\section{Boundary points when $\theta$ is inner}

We already know from Theorem~\ref{th:exposed and corners} that the exposed points of $\overline{W(S_\theta)}$ which are in $\bbD$ actually  belong to $W(S_\theta)$. More  can be said about boundary points of $W(S_\theta)$ when $\theta$ is inner. In this case one can describe all invariant subspaces of $S_\theta$ as having the form $\theta_1 \bm{H}_{\theta/\theta_1}$ for some inner divisor $\theta_1$ of $\theta$~\cite[Proposition X.2.5]{SNF}. Remember also that, by Theorem~\ref{th:exposed and corners}(1), an exposed point of $\overline{W(S_\theta)}$ belongs actually to $W(S_\theta)$.

\begin{theorem}\label{th:exposed1}
Assume $\theta$ is an inner function,  $w\in W(S_\theta)\cap\partial W(S_\theta)$, and $M\subset \Hth$ is a proper invariant subspace for $S_\theta$. Then $w\not\in W(S_\theta|M)$;
if, moreover, $w$ is an exposed point of $\overline{W(S_\theta)}$, then $w\not\in \overline{W(S_\theta|M)}$.
The set $\{x:\<S_\theta x,x\>=w\|x\|^2\}$ is a subspace of dimension one; if  $\xi_w$ generates it, then $\xi_w$ is cyclic for both $S_\theta$ and $S_\theta^*$.
\end{theorem}

\begin{proof}
Let  $d$ be a support line for $w$ which intersects $\bbT$ in $\alpha, \alpha'$. From Lemma~\ref{le:eigenvalues} and Theorem~\ref{th:facts theta} it follows that $\theta$ has an angular derivative in the sense of Cara\-th\'eo\-dory at $\alpha$ and $\alpha'$, and $\alpha, \alpha'$ are consecutive solutions of the equation $\widehat\theta(z)=\lambda$ for some $\lambda\in\bbT$. Thus, if we move from $\alpha$ to $\alpha'$ counterclockwise, the argument of $\widehat\theta(z)$ increases by $2\pi$.

If $M$ is an invariant subspace of $S_\theta$,  there exists a decomposition $\theta=\theta_1\theta_2$, with $\theta_i$ inner, such that $M=\theta_1 \bm{H}_{\theta_2}$; $S_\theta|M$ is then unitarily equivalent to $S_{\theta_2}$.
If $w\in W(S_\theta|M)$, then $d$ is also a support line for $w$ with respect to this last set.  A similar reasoning tells us that the argument of $z\theta_2(z)$ has to increase by $2\pi$ if we move from $\alpha$ to $\alpha'$ counterclockwise. But, since $\widehat\theta=z\theta_2\theta_1$, it would follow that the argument of $\theta_1$ stays constant, and thus $\theta_1$ is a constant. This contradicts the assumption that $M$ is a proper invariant subspace; thus $w\not\in W(S_\theta|M)$.

If $w$ is an exposed point of $\overline{W(S_\theta)}$ and belongs to  $\overline{W(S_\theta|M)}$, then it is also an exposed point of $\overline{W(S_\theta|M)}$. By Theorem~\ref{th:exposed and corners}(1), it must be actually in  $W(S_\theta|M)$, and the preceding argument applies. 

The subspace $Y=\{x\in H^2\ominus \theta_2 H^2:\<S_\theta x,x\>=w\|x\|^2\}$ has dimension~1 by Lemma~\ref{le:eigenvalues}. 
If $\xi_w$ generates $Y$, and $M$ is the invariant subspace of $S_\theta$ generated by $\xi_w$, then~(1) implies that $M=\Hth$. The same argument can be applied to $\bar w$, which is an exposed point in $W(S_\theta^*)$, if we note that $S_\theta^*$ is unitarily equivalent to $S_{\widetilde\theta}$, where $\widetilde{\theta}(z)=\overline{\theta(\bar z)}$.
\end{proof}

It follows from Theorem~\ref{th:exposed1}(1) that, if $M\subset \Hth$ is a proper invariant subspace for $S_\theta$, then $W(S_\theta|M)$ is contained in the interior of $W(S_\theta)$. For $\theta$  a finite Blaschke product, this is proved in~\cite{Ber}. Note that the assertion is not necessarily true for $\overline{W(S_\theta|M)}$, as shown by Example~\ref{ex:interior}. 

We end this section with the analogue of a finite dimensional result proved in~\cite{GauWu3, Ber}. Recall that a completely nonunitary contraction is said to be of class $C_0$ if $u(T)=0$ for some nonzero function $u\in H^\infty$.

\begin{theorem}\label{th:exposed direct summand} Assume that $\theta\in H^\infty$ is an inner function and
 $T\in\LL(\HH)$ is a contraction of class $C_0$, such that $\theta(T)=0$. Then we have $\overline{W(T)}\subset\overline{W(S_\theta)}$.  Moreover, if $W(T)$ contains a point in $W(S_\theta)\cap \partial W(S(\theta))$, then  $T$ has an orthogonal summand which is unitarily equivalent to $S_\theta$.
 \end{theorem}

\begin{proof}
It is well-known that if $\theta$ is the minimal function of $T$, then $T$ is unitarily equivalent to the restriction of $S_\theta\otimes 1_\KK\in\LL(\Hth\otimes\KK)$ (for some Hilbert space $\KK$) to an invariant subspace~$M$. We can assume that $\HH=M$ and $T$ is equal to this restriction. The inclusion $\overline{W(T)}\subset\overline{W(S_\theta)}$ is obvious.

Let $w\in W(S_\theta)\cap \partial W(S(\theta))$ be the given  point and take $h\in\HH$ with $\<Th,h\>=w$. We have $h=\sum_{i\in I} a_i( h_i\otimes e_i)$ for some unit vectors $h_i\in \Hth$, an orthonormal sequence $(k_i)_{i\in I}\subset\KK$, and $a_i> 0$ with $\sum a_i^2=1$. Since
\[
w=\<Th,h\>=\sum_{i\in I}a_i^2 \<S_\theta h_i, h_i\>
\]
is an exposed point, it follows that we must have $\<S_\theta h_i, h_i\>=w$ for all $i\in I$.

According to Theorem~\ref{th:exposed1}(2), we must have  $h_i=\gamma_i\xi_w$ for some $\gamma_i$ with $|\gamma_i|=1$. Then
\[
h=\sum_i a_i\gamma_i (\xi_w\otimes e_i)=\xi_w\otimes e,
\]
where $e=\sum_i a_i\gamma_i e_i\in \KK$ is a unit vector. Again by Theorem~\ref{th:exposed1}(2), $\xi_w$ is cyclic for $S_\theta$ and $S_\theta^*$, so $\Hth\otimes e\subset \HH$ is a reducing subspace of  $S_\theta\otimes 1_\KK$ contained in~$\HH$ which provides the required orthogonal summand.
\end{proof}

\section{Examples}\label{se:examples}

We  consider some examples of functions $\theta$ in the unit ball of $H^\infty$ and  discuss the numerical ranges of the corresponding model operators $S_\theta$. In particular, they show that all  cases described in Theorems~\ref{th:one-to-one on the arc},~\ref{th:angle at end point}, and~\ref{th:end points}  actually occur. Also,  it was proved for finite Blaschke products in~\cite{GauWu2} and for inner functions with one singularity on the unit circle in~~\cite{CGP2} that $\overline{W(S_\theta)}$  determines the function~$\theta$.
Examples~\ref{ex:conformal} and~\ref{ex:interior} show that this is not true in general.

\begin{example}\label{ex:conformal}
Let $\theta$ be a conformal mapping of the disc $\bbD$ onto the region $G$ bounded by an arc $\wideparen{z_1z_2}$ on $\bbT$ and a smooth curve that joins $z_1$ and $z_2$ and is contained in $\bbD$. If $\theta(\lambda_i)=z_i$, then $\sigma(T)\cap\bbT$ is the closure of the arc $\wideparen{\lambda_1\lambda_2}$, where $\theta(\wideparen{\lambda_1\lambda_2})=\wideparen{z_1z_2}$, while it follows from Theorem~\ref{th:one-to-one on the arc} that $\overline{W(T)}$ is precisely the closed convex hull of this set; that is, the circular segment determined by $\wideparen{\lambda_1\lambda_2}$.
Note that this numerical range is depends only on $\lambda_1$ and $\lambda_2$, and  not  on the precise shape of~$G$.
\end{example}

The example above has the advantage of simplicity, but one can obtain interesting examples of numerical ranges by looking only at inner functions $\theta$. We may thus discuss the different cases possible cases for a point in the boundary of $\overline{W(T)}\cap\bbT$.

\begin{example}\label{ex:infinite solutions}
Let $\theta=BS$ be the decomposition into Blaschke product $B$ and singular inner part with measure $\mu$, and let $\zeta\in\bbT$ be an endpoint of  a connected component of $\bbT\setminus \sigma(S_\theta)$. It is shown in~\cite[Lemma 3.3]{CGP3} that if there exists a sequence of zeros of $B$ that converge to $\zeta\in\bbT$ in a Stolz angle, or if $\mu(\{\zeta\})>0$, then $\widehat\theta$ is not one-to-one on any arc $\wideparen{\zeta\zeta'}$. The assumptions of Proposition~\ref{pr:limit of the function} are not satisfied, and $\partial W(T)$ is smooth at $\zeta$ by Theorem~\ref{th:angle at end point}.
\end{example}

\begin{example}\label{ex:corners with segment}
To obtain examples of corners of $\partial W(T)$, we use the above quoted fact that if $\Wth$ is written under the form~\eqref{eq:decomposition of function}, then~\eqref{eq:angular derivative} is necessary and sufficient for the existence of the angular derivative.
Choose then a function $\theta$ such that $\zeta$ is an endpoint of a component of $\bbT\setminus\sigma(S_\theta)$ and~\eqref{eq:angular derivative} is satisfied. According to Theorem~\ref{th:end points}, $\partial W(T)$ contains then a segment with one endpoint at~$\zeta$.
\end{example}

\begin{example}\label{ex:corners no segment}
Another result of Ahern and Clark~\cite{AC1} says that if an inner function is written as in~\eqref{eq:decomposition of function} with $\nu$ singular, $\nu(\{\zeta\})=0$, and
\begin{equation}\label{eq:radial limit}
\sum_n\frac{1-|a_n|}{|\zeta-a_n|} + \int_{[0,2\pi)} \frac{d\nu(t)}{|e^{it}-\zeta|} <\infty,
\end{equation}
then $\theta$ has a nontangential limit at $\zeta$. Choose the zeros of the Blaschke product and the singular measure such that $\theta$ is an endpoint of a connected component of $\bbT\setminus\sigma(S_\theta)$,~\eqref{eq:radial limit} is satisfied, but~\eqref{eq:angular derivative} is not. Then, according to Theorem~\ref{th:end points}, $\zeta$ is a corner of $\partial W(T)$, but $\partial W(T)$ does not contain a line segment ending at $\zeta$.
 \end{example}

\begin{example}\label{ex:interior}
The last example is related to Theorem~\ref{th:exposed1}. Consider a Blaschke product $\theta$ whose zeros $a_n$ accumulate on $\{z\in\bbT: \Im z\le 0\}$. This can be done such that the left hand side of~\eqref{eq:angular derivative} is as small as we like, both for $\zeta=1$ and for $\zeta=-1$. If this quantity is sufficiently small, $\theta$ is one-to-one on $\{z\in\bbT: \Im z> 0\}$, and it follows then from Theorem~\ref{th:one-to-one on the arc} that $\overline{W(S_\theta)}=\{z\in\bbC:|z|\le 1, \Im z\le 0\}$.

 We can then take a proper subsequence of  $a_n$ with the same properties. If $\theta'$ is the corresponding Blaschke product, then $M=\frac{\theta}{\theta'} \bm{H}_{\theta'}$ is a proper invariant subspace of $S_\theta$, and its restriction to this subspace is unitarily equivalent to $S_{\theta'}$. We know by Theorem~\ref{th:exposed1} that $W(S_{\theta'})$ is contained in the interior of $W(S_{\theta})$.  However, $\overline{W(S_\theta|M})=\overline{W(S_{\theta'})}= \overline{W(S_{\theta})}=\{z\in\bbC:|z|\le 1, \Im z\le 0\}$. 
 
We can be more careful with the choice of the zeros of $\theta$, such that $\Wth$ has angular derivatives in the sense of Carath\'eodory at 1 and $-1$, and  $\Wth(1)=\Wth(-1)$. Then the argument of $\Wth$ increases by $2\pi$ on the upper semicircle; also, this is not possible for the argument of $\widehat\theta'$, which is a proper divisor of $\Wth$. It follows from Lemma~\ref{le:eigenvalues} and Theorem~\ref{th:facts theta} that $W(S_\theta)$ contains precisely one point $w\in [-1,1]$, while $W(S_{\theta'})\cap\bbR=\emptyset$. Thus
\[
W(S_{\theta'})=\{z\in\bbC:|z|< 1, \Im z< 0\}\subsetneq \{z\in\bbC:|z|< 1, \Im z< 0\}\cup\{w\}= W(S_\theta).
\]
 \end{example}

\end{document}